\documentclass[preprint]{elsarticle}
\RequirePackage{amsthm,amsmath,amssymb}

\newcommand{\R}{{\mathbb R}}

\numberwithin{equation}{section}

\newtheorem{theorem}{Theorem}[section]

\theoremstyle{remark}

\newtheorem{corollary}{Corollary}[section]
\theoremstyle{definition}

\journal{Statistics and Probability Letters}
\begin{document}
\begin{frontmatter}
\title{Standard maximum likelihood drift parameter estimator in the homogeneous
diffusion model is always strongly consistent\tnoteref{T1}}


\author{Yuliya Mishura}
\ead{myus@univ.kiev.ua}



\address{ Taras Shevchenko National University of Kyiv,
Mechanics and Mathematics Faculty, Volodymyrska 60, 01601 Kyiv,
Ukraine}

\begin{abstract}
We consider the homogeneous stochastic differential equation with
unknown parameter to be estimated.  We prove that the standard
maximum likelihood estimate is   strongly consistent under very
mild conditions.  There are also established the conditions for
strong consistency of the discretized estimator.
\end{abstract}

\begin{keyword}
Stochastic differential equation with homogeneous coefficients\sep
drift parameter\sep strong consistency\sep discretized model

\MSC[2010] 62F12 \sep 62M05 \sep 60H10 \sep 60J60
\end{keyword}

\end{frontmatter}

\section{Introduction}

There is an extended  literature devoted to standard and
nonstandard approaches to the drift parameter estimation in the
diffusion models, both for discrete and continuous observations.
We mention only the books \cite{lish}, \cite{prakasa rao},
\cite{hey}, \cite{soren} and references therein.  Many complicated
models have been studied.  However, there was a curious gap even
in the case of simplest homogeneous diffusion model:  there were
no conditions for the strong consistency of the standard maximum
likelihood estimator that are  close to be necessary and are
sufficiently mild. We have filled the gap, applying the results of
the paper \cite{miur} and have proved that, in some sense, the
standard maximum likelihood estimator is always strongly
consistent unless the drift coefficient is identically zero.

The paper is organized as follows.  In Section 2 we prove that the
denominator in the stochastic representation of the maximum
likelihood estimator tends to infinity under very mild conditions
and deduce from here the strong consistency. In Section 3 we
establish the sufficient conditions for the strong consistency of
the discretized version of the maximum likelihood estimator. Some
simulation results are included.

\section{Preliminaries}

Let $(\Omega,\Im,\{\Im_{t}\}_{t\geq 0},P)$ be a complete
probability space with filtration that satisfies the standard
conditions. Let $W=\{W_t,\Im_t,t\geq0\}$ be a standard Wiener
process. Consider a homogenous diffusion process $X=\{X_t,\Im_t,t\geq0\}$
that is a  solution to the stochastic differential equation

\begin{equation}\label{DiffEq}
X_{t}=x_{0}+\theta\int\limits_0^t a(X_{s})ds+\int\limits_0^t
b(X_s)dW_{s}.
\end{equation}
Here $x_{0}\in\mathbb{R}$; $\theta \in \mathbb{R}$ is  unknown
parameter  to be estimated, $a,b:\mathbb{R}\rightarrow\mathbb{R}$
are measurable  functions, $b(x)\neq0$ for any $x\in\mathbb{R}$, $a$ is not zero identically.
In general, we only need the existence and uniqueness of the weak
solution of equation \eqref{DiffEq} on the whole axis. Recall that
any of the  following groups of conditions  on $a$ and $b$
supplies the existence-uniqueness for the strong solution:

Yamada conditions (\cite{yamada}, \cite{ikeda-watan}):

\begin{itemize}
\item[](i) Linear growth: there exists $K>0$ such that for any $x\in\mathbb{R}$

$|a(x)|+|b(x)|\leq K(1+|x|)$;

\item[](ii) There exists such convex increasing function
$k:\mathbb{R}_+\rightarrow\mathbb{R}_+$ that
$k(0)=0$,
$\int_{0+}k^{-1}(u)du=+\infty$ and for any $x,y\in\mathbb{R}$
$|a(x)-a(y)|\leq k(|x-y|)$;

\item[](iii) There exists such strictly increasing function
$\rho:\mathbb{R}_+\rightarrow\mathbb{R}_+$ that $\rho(0)=0$,
$\int_{0+}\rho^{-2}(u)du=+\infty$ and for any $x,y\in\mathbb{R}$
$|b(x)-b(y)|\leq \rho(|x-y|)$.

\end{itemize}

Krylov--Zvonkin  conditions \cite{kryl zvon}:

\begin{itemize}
\item[](i) Function $a$ is bounded, function $b$ is separated from $0$: $b(x)\geq \alpha>0, x\in\mathbb{R}$;

\item[](ii) Function $b$ has locally bounded variation:  for any $N>0$ $$var_{[-N,N]}
b:=\sup_{-N=x_{0}<x_{1}<...<x_{n}=N}\sum|b(x_{k+1})-b(x_{k})|<\infty.$$
\end{itemize}

Existence  of the weak solution of equation \eqref{DiffEq}
holds under continuity and linear growth of the coefficients.
It was initially proved in \cite{skor}. Then the conditions
of existence and uniqueness  of the weak solution were generalized
in \cite{krylov} and the most general conditions were obtained in
\cite{engel schmidt} and \cite{engel schmidt2}.

\section{Strong consistency of the drift parameter maximum-likelihood estimator constructed for continuous observations}

Denote the functions $c(x)=\frac{a(x)}{b^{2}(x)}$,
$d(x)=\frac{a^2(x)}{b^{2}(x)}.$ In what follows we suppose that
the following condition holds:
$$(A)\;\;  \text{functions}\;\;  \frac{1}{b^{2}}, \;\; d\;\; \text{and}\;\;  \frac{d}{b^{2}}\;\;  \text{are locally integrable}.$$

 Furthermore, denote
$L_{t}(x)=b^{2}(x)\lim_{\varepsilon\downarrow\infty}\frac{1}{2\varepsilon}\int\limits_0^t1\{|X_{s}-x|<\varepsilon\}ds$
the local time of the process $X$ at the point $x$ on the interval
$[0,t]$, $t\geq0$. Then, according, e.g., to \cite{pitman-yor},
for any locally integrable function $f$ the following equality
holds:
$$\int\limits_0^t
f(X_{s})ds=\int_{\mathbb{R}}\frac{f(x)}{b^2(x)}L_{t}(x)dx, t\geq
0.$$ Therefore, under the condition of local integrability,
$\int\limits_0^td(X_{s})ds<\infty$ for any $t>0$. As it is
well-known, a likelihood function for  equation \eqref{DiffEq} has
a form

\begin{multline*}
\frac{dP_{\theta}(t)}{dP_{0}(t)}=\exp\bigg\{\theta\int\limits_0^t\frac{a(X_{s})}{b(X_{s})}dW_{s}+
\frac{\theta^{2}}{2}\int\limits_0^td(X_{s})ds\bigg\}=
\exp\bigg\{\theta\int\limits_0^tc(X_{s}) dX_{s}\\
-\frac{\theta^{2}}{2}\int\limits_0^td(X_{s}) ds\bigg\},
\end{multline*}
and the maximum likelihood estimator of parameter $\theta$
constructed by the observations of $X$ on the interval $[0,t]$,
has a form

\begin{equation}\label{LikEst}
\hat{\theta}_{t}=\frac{\int\limits_0^tc(X_{s})dX_{s}}{\int\limits_0^td(X_{s})ds}=
\theta+\frac{\int\limits_0^t\frac{a(X_{s})}{b(X_{s})}dW_{s}}{\int\limits_0^t d(X_{s})d{s}}.
\end{equation}
In order to establish the criteria of the strong consistency of
$\hat{\theta}_{t}$ in terms of the coefficients $a$ and $b$,
denote
$\varphi(x)=\exp\Big\{-2\theta\int\limits_0^xc(y)dy\Big\}$,
$\Phi(x)=\int\limits_0^x \varphi(y)dy$. Concerning the asymptotic behavior of the integral $\int\limits_0^t d(X_{s})d{s}$ under the fixed value of parameter $\theta\neq 0$,  two cases can
be considered.

Let for some $\theta \in \R$ $\Phi(+\infty)=-\Phi(-\infty)=+\infty$. In this case the diffusion
process $X$ is recurrent and its trajectories have the property:
$\overline{\lim}_{t\to\infty}X_{t}=+\infty$ a.s. and
$\underline{\lim}_{t\to\infty}X_{t}=-\infty$ a.s.

Furthermore, $\int\limits_0^\infty
f(X_{s})ds=\int_{\mathbb{R}}\frac{f(x)}{b^2(x)}L_{\infty}(x)dx$. However, as it
was mentioned in \cite{pitman-yor} and  \cite{itomackean}, $L_{\infty}(x)=\infty$ $P$-a.s. for any
$x\in\mathbb{R}$ and recurrent process $X$. It means that
$\int\limits_0^\infty f(X_{s})ds=\infty$ a.s. for any $f$ that is not identically $0$,
and in this case
\begin{equation}\label{Eqq}
\int\limits_0^\infty d(X_{s})ds=\infty \quad P-\text{a.s}.
\end{equation}

Now, let at least one of the integrals $\Phi(+\infty)$ or
$\Phi(-\infty)$ be finite. In this case the process $X$ is
transient. We shall apply the following result that is the
reformulation of Theorem 2.12 from  \cite{miur}.  Denote
$\psi_+(x)=\frac{\int\limits_x^{+\infty}\varphi(y)dy
}{\varphi(x)}$,
$\psi_-(x)=\frac{\int\limits_{-\infty}^x\varphi(y)dy
}{\varphi(x)}$, $J_\infty (f)=\int\limits_0^{+\infty} f(X_{s}) ds.$
\begin{theorem}\label{miur}(\cite{miur})\begin{itemize}

\item Let $\psi_+(0)<\infty, \psi_-(0)=\infty$.

If   $I_1(f):=\int\limits_0^{+\infty} \frac{|f(x)|}{a^2(x)}\psi_+(x)dx
<\infty $, then $J_\infty(f)\in \R $ $P$-a.s.

If $I_1(f)
=\infty $ then $J_\infty(f)=\infty $ $P$-a.s.
\item Let $\psi_+(x)= \infty, \psi_-(x)<\infty$.

If  $I_2(f):=\int\limits^0_{-\infty} \frac{|f(x)|}{a^2(x)}\psi_-(x)dx
<\infty $, then $J_\infty(f)\in \R $ $P$-a.s.

If $I_2(f)
=\infty $ then $J_\infty(f)=\infty $ $P$-a.s.
\item Let $\psi_+(x)< \infty, \psi_-(x)<\infty$.

If   $I_1(f)
<\infty $, then $J_\infty(f)\in \R $ $P$-a.s. on the set $X_s^x\rightarrow +\infty$.

If $I_1(f)
=\infty $ then $J_\infty(f)=\infty $ $P$-a.s. on the set $X_s^x\rightarrow +\infty$.

If $I_2(f)
<\infty $, then $J_\infty(f)\in \R $ $P$-a.s. on the set $X_s^x\rightarrow -\infty$.

If $I_2(f)
=\infty $ then $J_\infty(f)=\infty $ $P$-a.s. on the set $X_s^x\rightarrow -\infty$.
\end{itemize}
\end{theorem} 
%
\begin{theorem}\label{thdmvd}
(1) Let for some $\theta\neq 0$  $\Phi(+\infty)< +\infty $. Then  $$I_{1}(d)=\int\limits_0^{+\infty}\frac{d(x)}{
b^{2}(x)}\psi_+(x)dx
=+\infty .$$

%

(2) Let for some $\theta\neq 0$ $\Phi(-\infty)< \infty $.
Then  $I_{2}(d):=\int\limits_{-\infty}^0 \frac{d(x)}{
b^{2}(x)}\psi_-(x)dx
=\infty $.
\end{theorem}
%
%
%
%
%
%
%
\begin{proof} We prove only the first statement since the second one can be proved similarly. Note that $\frac{\varphi(y)}{\varphi(x)}=\exp\Big\{-2\theta\int_x^y c(u)du\Big\}$ and $\frac{d(x)}{
b^{2}(x)}=\theta^2c^2(x)$. It means that without loss of generality, we can put $\theta=1$. Therefore, applying Fubini theorem for nonnegative integrands and Schwartz inequality, we get
\begin{equation}\begin{gathered}\label{bound}
I_1(d)=\int_0^{\infty}c^2(x)\int_x^{\infty}\exp\Big\{-2\int_x^y c(u)du\Big\}dy dx\\
=\int_0^{\infty}\int_0^y c^2(x)\exp\Big\{-2\int_x^y c(u)du\Big\}dx dy\\
\geq \int_0^{\infty}\Big(\int_0^y c (x)\exp\Big\{- \int_x^y c(u)du\Big\}dx\Big)^2 \frac{dy}{y}\\
\geq \int_1^{\infty}\Big(\int_1^y c (x)\exp\Big\{- \int_x^y c(u)du\Big\}dx\Big)^2 \frac{dy}{y}\\=
\int_1^{\infty}\Big(1-\exp\Big\{- \int_1^y c(u)du\Big)^2\frac{dy}{y}.
\end{gathered}
\end{equation}
It is sufficient to prove that the last integral in \eqref{bound} diverges. However, it consists of three terms, one of which, $\int_1^{\infty}\frac{dy}{y}$ diverges, and two other converge: $$\int_1^{\infty}\exp\Big\{- \int_1^y c(u)du\Big\}\frac{dy}{y}\leq (\Phi(\infty))^{\frac12}\Big(\int_0^{\infty}\frac{dy}{y^2}\Big)^{\frac12}<\infty$$
and $$\int_1^{\infty}\exp\Big\{- 2\int_1^y c(u)du\Big\}\frac{dy}{y}\leq \Phi(\infty)<\infty.$$
\end{proof}
\begin{corollary}\label{cor1}
As an immediate consequence of Theorems \ref{miur}, \ref{thdmvd}
and formula \eqref{Eqq},  we get the following statement: let
equation \eqref{DiffEq} have the weak solution, the coefficients
$a$ and $b$ satisfy the condition $(A)$ and $a$ be not identically
zero. Then $\int_0^{\infty}d(X_s)ds=+\infty$ $P$-a.s.
\end{corollary}
The next theorem  is the main result in this section.

%
\begin{theorem}\label{thdmvd2}

Let equation \eqref{DiffEq} has the weak solution, coefficients
$a$ and $b$ satisfy condition $(A)$  and $a$ is not identically
zero. Then maximum likelihood estimator  $\hat{\theta}_{t}$ is
strongly consistent  as $t\rightarrow\infty $.
\end{theorem}

\begin{proof}
According to representation (\ref{LikEst}), $\hat{\theta}_{t}$ is
strongly consistent if for locally square-integrable martingale
$M_{t}=\int\limits_0^t \frac{a (X_{s})}{ b (X_{s})}dW_{s} $ we
have that $\frac{M_{t}}{\langle M\rangle_{t}}\rightarrow  0$ $P$-a.s. However,
according to the strong law of large numbers for martingales
(Theorem 10, \S 6, Chapter 2, \cite{liptser-shiryaev1}), under condition
$\langle M\rangle_{t}\rightarrow \infty, t\rightarrow \infty $ P-a.s., we have
that $\frac{M_{t}}{\langle M\rangle_{t}}\rightarrow  0$ P-a.s. The proof
immediately follows now from Corollary \ref{cor1}.
\end{proof}

\section{Discretization and strong consistency}

In this section we suppose that the coefficients $a$, $b$ and $c$ are
bounded and Lipschitz, more precisely, satisfy condition: for some $a_0>0$ and $K>0$ and for any $x,y\in \mathbb{R}$

$$(B)\;\;|a(x)|+|b(x)|+|c(x)|\leq a_0,\;\;
|a(x)-a(y)|+|b(x)-b(y)| \le K |x-y| .$$
Let $0<
\alpha<\frac{1}{2}$. Suppose that we observe the process $X$ that is the solution of
equation (\ref{DiffEq}), only at discrete moments of time
$t_{k}^{n}=\frac{k}{n}, 0\le k\le n^{1+\alpha}$.
Consider a discretized version of the estimate $\hat{\theta}_{t}$:
\begin{equation*}
\hat{\theta}_{n}^{d}=\frac{\sum\limits_{k=0}^{n^{1+\alpha}}c(X_{ \frac kn})\bigtriangleup
X_{k}^{n}}{\frac{1}{n}\sum\limits_{k=0}^{n^{1+\alpha}}d(X_{ \frac kn})},
\end{equation*}
where $\bigtriangleup X_{k}^{n}= X_{ \frac {k+1}n}- X_{
\frac kn}$. Then
\[
\hat{\theta}_{n}^{d}=\Bigg(\frac{1}{n}\sum\limits_{k=0}^{n^{1+\alpha}}d(X_{
\frac kn}) \Bigg)^{-1}\Bigg(\sum\limits_{k=0}^{n^{1+\alpha}}d(X_{
\frac kn}) \Bigg(\theta \int\limits_{ \frac kn}^{ \frac
{k+1}n}a(X_{s})ds+\int\limits_{ \frac kn}^{ \frac
{k+1}n}b(X_{s})dW_{s}\Bigg)\Bigg)=\]
\[
=\theta+\Bigg(\frac{1}{n}\sum\limits_{k=0}^{n^{1+\alpha}}d(X_{
\frac kn}) \Bigg)^{-1}\Bigg(\sum\limits_{k=0}^{n^{1+\alpha}}c(X_{
\frac kn}) \theta \int\limits_{ \frac kn}^{ \frac
{k+1}n}(a(X_{s})-a(X_{ \frac kn}))ds+\]
\[
+\sum\limits_{k=0}^{n^{1+\alpha}}\frac{a (X_{ \frac kn})}{ b (X_{
\frac kn})}\bigtriangleup W_{
k}^n+\sum\limits_{k=0}^{n^{1+\alpha}}c(X_{\frac kn})
\int\limits_{\frac kn}^{ \frac {k+1}n}(b(X_{s})-b(X_{ \frac
kn}))dW_{s}\Bigg)=:\theta +I_{1}^{n}+ I_{2}^{n}+ I_{3}^{n}.
\]

\begin{theorem}\label{thdmvd3}
Let equation \eqref{DiffEq} has the weak solution,  coefficients $a$ and $b$ satisfy condition $(B)$ and $a$ is not identically zero.
Then $\hat{\theta}_{n}^{d}$ is strongly consistent as $n \to
\infty$.

\end{theorem}

\begin{proof}
It is sufficient to prove that $I_{r}^{n}\to 0$, $r=1,2,3$ a.s. as
$n\to \infty$. Evidently, the denominator $\Big(\frac{1}{n}\sum\limits_{k=0}^{n^{1+\alpha}}d(X_{
\frac kn}) \Big)^{-1}$ tends to infinity a.s. as $n\rightarrow \infty$. Consider the numerator, say  $J_1^n$, for $I_{1}^{n}$:

\begin{equation*}\begin{gathered}
|J_1^n|=\Big|\sum_{k=0}^{n^{1+\alpha}}c(X_{ \frac kn})\theta \int\limits_{ \frac kn}^{
\frac {k+1}n} (a(X_{s})-a(X_{ \frac kn}))ds\Big|\\ \le
{|\theta|a_{0}}n^{\alpha}\sup_{0\le k\le
n^{1+\alpha}, \frac kn\le s\le \frac {k+1}n}|a(X_{s})-a(X_{
\frac kn})|\\ \le
{K|\theta|a_{0}}n^{\alpha}\sup_{0\le k\le
n^{1+\alpha}, \frac kn\le s\le \frac {k+1}n}|X_{s}- X_{
\frac kn}|.
\end{gathered}\end{equation*}
In turn, $|X_{s}-X_{\frac kn}|\le \frac{a_{0}}{n}+|\int\limits_{
\frac kn}^{ s}b(X_{u})dW_{u}|$.
Therefore, $$J_{1}^{n}\le {K|\theta|a_{0}^{2}}
n^{\alpha-1}+{K|\theta|a_{0}}\xi_n,$$ where $\xi_n=
n^{\alpha}\sup_{0\le k\le n^{1+\alpha}, \frac kn\le s\le
\frac {k+1}n}|\int\limits_{ \frac kn}^{ s}b(X_{u})dW_{u}|$.
For any $\varepsilon >0$ denote $A_{n}=\{  \omega:
\xi_n\ge
\varepsilon\}$.
Then it follows from Burkholder-Gundy inequality that for any
$p>1$ $$P(A_{n})\le C_{p}\varepsilon ^{-p}n^{\alpha
p}\sum_{k=1}^{n}E(\int\limits_{ \frac kn}^{
\frac {k+1}n}|b(X_{u})|^{2}du)^{\frac{p}{2}}\le
C_{p}a_{0}^{p}\varepsilon^{-p}n^{\alpha p-\frac{p}{2}+1},$$ and
$\sum_{n=1}^{\infty}P(A_{n})<+\infty$ if we choose
$p>\frac{4}{1-2\alpha}$.
It means that for any $\omega \in \Omega$ there exists
$n(\omega)$ such that for $n>n(\omega)$
$$n^{\alpha}\sup_{0\le k\le n^{1+\alpha}, \frac kn\le s\le
\frac {k+1}n}|\int\limits_{ \frac kn}^{ s}b(X_{n})dW_{u}|\le
\varepsilon .$$ 
%
Therefore, $I_{1}^{n}\to 0$, $n\to\infty$ P-a.s.
Consider the term $I_{2}^{n}$. Denote martingale
$N_{n}:=\sum_{k=0}^{n^{1+\alpha}} \frac{a(X_{\frac kn})}{ b (X_{
\frac kn})}\bigtriangleup W_{ {k}}^{n}$. Then
$I_{2}^{n}=\frac{N_{n}}{\langle N\rangle_{n}}\to 0, t\to \infty$ a.s. since $\langle
N\rangle_{t}\to\infty$ a.s.

Consider the numerator for $I_{3}^{n}$. It is the
square-integrable martingale with respect to the discretized
filtration $\{\Im_{\frac kn}=\sigma\{X_{\frac in}$, $0\leq i\leq
k\}$, $0\leq k\leq n^{1+\alpha}\}.$  Denote it sa $P_{n}$. Its
quadratic characteristic equals
$$\langle P\rangle_{n}=\sum_{k=0}^{n^{1+\alpha}}{
c^2(X_{\frac kn})}\int\limits_{ \frac kn}^{
\frac {k+1}n}E((b(X_{s})-b(X_{ \frac kn}))^{2}/\Im_{\frac kn})ds,$$ and

\begin{equation*}\begin{gathered}\langle P\rangle_{n}\le {a_{0}^{2}}n^{\alpha}\sup_{0\le k\le
n^{1+\alpha}, \frac kn\le s\le
\frac {k+1}n}E((b(X_{s})-b(X_{ \frac kn}))^{2}/\Im_{\frac kn})\\\le
{2K^{2}a_{0}^{2}}n^{\alpha-2}+{2K^{2}a_{0}^{2}}n^{\alpha-1}\le
C{n}^{\alpha-1}\end{gathered}\end{equation*} with some constant $C>0$. Now we use the fact from \cite{liptser-shiryaev1} that for any locally square
integrable martingale $Y$ and for any constant $a>0$ $\frac{Y_t}{a+\langle Y\rangle_t}$ converges a.s. to some
finite random variable as $t\rightarrow \infty$.  Therefore, we can take some $a>0$ and conclude  that
$$\frac{P_{n}}{a+\langle P\rangle_{n}}{\longrightarrow} \xi\;\;\text{a.s.},$$ where $\xi$ is some random variable and consequently
\begin{equation*}\begin{gathered}\frac{P_{n}}{\frac 1n\sum_{k=1}^{n^{1+\alpha}}d(X_{
\frac kn})}=\frac{P_{n}}{a+\langle P\rangle_{n}}\cdot\frac{a+\langle P\rangle_{n}}{\frac 1n\sum_{k=1}^{n^{1+\alpha}}{ d(X_{ \frac kn})}}\to 0\end{gathered}\end{equation*} a.s. as $n\to
\infty$.
\end{proof}

\section{Some simulation results}

We simulated the model and set the  discretization interval
$\Delta t=0.01$; number of the simulated trajectories is $1 000$;
the value of the parameter to be  estimated equals $1$. Let us
consider three cases:
\begin{itemize}
\item[(i)] Let $ a(x)= 1+x$, $b(x)=x^{-1/3}$. Then for different $t$
we have the values of $\hat{\theta}_{t}$ as presented in the
Table~1.
\begin{flushleft}
\textit{Table~1}
\end{flushleft}

\begin{center}
\begin{tabular}{|c|c|c|c|c|c|c|}
  \hline
  $t$ & 1 & 10 & 20 & 30 \\
  \hline
 $\hat{\theta}_{t}$ & 0.870 & 0.999 & 1 - 5$\times 10^{-8}$ & 1 + $10^{-11}$ \\
  \hline
\end{tabular}
\end{center}

\item[(ii)] Let $a(x)= 1+x$, $b(x)=2+\sin x$. Then for different $t$
we have the values of $\hat{\theta}_{t}$ as presented in the
Table~2. \pagebreak

\begin{flushleft}
\textit{Table~2}
\end{flushleft}
\begin{center}
\begin{tabular}{|c|c|c|c|c|c|c|}
  \hline
  $t$ & 5 & 10 & 20 & 30 \\
  \hline
 $\hat{\theta}_{t}$ & 0.908 & 0.997 & 1 + 2$\times 10^{-7}$ & 1 + 6$\times10^{-11}$ \\
  \hline
\end{tabular}
\end{center}

\item[(iii)] Let $ a(x)=|x|1_{\{|x|\leq 1\}}$, $b(x)=1$. Then for
different $t$ we have the values of $\hat{\theta}_{t}$ as
presented in the Table~3.
\begin{flushleft}
\textit{Table~3}
\end{flushleft}
\begin{center}
\begin{tabular}{|c|c|c|c|c|c|c|c|}
  \hline
  $t$ & 5 & 10 & 50 & 100 & 500 & 1 000&10 000 \\
  \hline
 $\hat{\theta}_{t}$ & 1.962 & 1.67 & 1.35 & 1.31 & 1.19 & 1.1& 1.07 \\
  \hline
\end{tabular}
\end{center}
\end{itemize}
We see that in the last case, when the process is recurrent, the convergence is slow. It can be  explained in such a way: the drift coefficient ``often'' equals zero. When it is zero, we can not estimate the value of parameter. So, we must wait until sufficient quantity of information comes.
\bibliographystyle{elsarticle-num}

\end{document}